\documentclass[a4paper]{article}
\usepackage{hyperref}
\usepackage{graphicx}
\usepackage{mathrsfs}
\usepackage{enumerate}
\usepackage{amsxtra,amssymb,latexsym, amscd,amsthm,amsmath}
\usepackage{indentfirst}
\usepackage{color}
\usepackage[utf8]{inputenc}
\usepackage[mathscr]{eucal}
\usepackage{amsfonts}
\usepackage{graphics}
\usepackage{multirow}
\usepackage{array}
\usepackage{subfigure}
\usepackage{cite}
\usepackage{wrapfig}
\usepackage{tikz}
\usepackage{diagbox}

\usepackage{pgfplots}
\pgfplotsset{compat=1.15}
\usepackage{mathrsfs}
\usetikzlibrary{arrows}

\newcommand{\footremember}[2]{%
    \footnote{#2}
    \newcounter{#1}
    \setcounter{#1}{\value{footnote}}%
}
 
\def\R{{\mathbb R}}
\def\C{{\mathbb C}}
\def\N{{\mathbb N}}

\DeclareMathOperator{\lc}{lc}

\DeclareMathOperator{\inter}{int}

\newtheorem{theorem}{\bf Theorem}%[section]
\newtheorem{lemma}{\bf Lemma}
\newtheorem{algorithm}{\bf Algorithm}%[section]
\newtheorem{example}{\bf Example}%[section]
\newtheorem{proposition}{\bf Proposition}%[section]
\newtheorem{corollary}{\bf Corollary}%[section]
\newtheorem{definition}{\bf Definition}%[section]
\newtheorem{remark}{\bf Remark}%[section]
\providecommand{\keywords}[1]
{
  \small	
  \textbf{\textbf{Keywords:}} #1
}
\begin{document}
\definecolor{qqzzff}{rgb}{0,0.6,1}
\definecolor{ududff}{rgb}{0.30196078431372547,0.30196078431372547,1}
\definecolor{xdxdff}{rgb}{0.49019607843137253,0.49019607843137253,1}
\definecolor{ffzzqq}{rgb}{1,0.6,0}
\definecolor{qqzzqq}{rgb}{0,0.6,0}
\definecolor{ffqqqq}{rgb}{1,0,0}
\definecolor{uuuuuu}{rgb}{0.26666666666666666,0.26666666666666666,0.26666666666666666}
\newcommand{\vi}[1]{\textcolor{blue}{#1}}
\newif\ifcomment
\commentfalse
\commenttrue
\newcommand{\comment}[3]{%
\ifcomment%
	{\color{#1}\bfseries\sffamily#3%
	}%
	\marginpar{\textcolor{#1}{\hspace{3em}\bfseries\sffamily #2}}%
	\else%
	\fi%
}

\newcommand{\mapr}[1]{{{\color{blue}#1}}}
\newcommand{\revise}[1]{{{\color{blue}#1}}}
\newcommand{\victor}[1]{{{\color{red}#1}}}

\title{Semi-algebraic description of the closure of  the image of a semi-algebraic set under a polynomial}

\author{%
%Tien-Son Pham\footremember{2}{Department of Mathematics, Dalat University, 1 Phu Dong Thien Vuong, Dalat, Viet Nam.}\qquad 
Ngoc Hoang Anh Mai\footremember{1}{Independent scholar, Toulouse, France.}
%Dinh Sy Tiep\footremember{3}{Institute of Mathematics, VAST, 18, Hoang Quoc Viet Road, Cau Giay District 10307, Hanoi, Viet Nam.}\\
  %Jean-Bernard Lasserre\footrecall{1} \footrecall{2}
  }
%\date{March 29, 2019}

\maketitle

\begin{abstract}
Given a polynomial $f$ and a semi-algebraic set $S$, we provide a symbolic algorithm to find the equations and inequalities defining a semi-algebraic set $Q$ which is identical to the closure of the image of $S$ under $f$, i.e., 
\begin{equation}
Q=\overline{f(S)}\,.
\end{equation}
Consequently, every polynomial optimization problem whose optimum value is finite has an equivalent form with attained optimum value, i.e., 
\begin{equation}
\min \limits_{t\in Q} t =\inf\limits_{x\in S} f(x)
\end{equation}  
whenever the right-hand side is finite.
Given $d$ being the upper bound on the degrees of $f$ and polynomials defining $S$, we prove that our method requires $O(d^{O(n)})$ arithmetic operations to produce polynomials of degrees at most $d^{O(n)}$ that define $\overline{f(S)}$.
%Following our algorithm, we obtain a new type of sums-of-squares representation for every polynomial $h$ nonnegative on $S$ under the boundedness of the image $h(S)$.
\end{abstract}
\tableofcontents
\section{Introduction}
Let $\R[x]$ stand for the ring of polynomials of real coefficients in vector of variables $x=(x_1,\dots,x_n)$.
The set $W\subset \R^n$ is called elementary semi-algebraic, if it can be written in the form
\begin{equation}
W = \{x\in\R^n\,:\,p_1(x)> 0,\dots,p_m(x)> 0\,,\,q_1(x)=\dots=q_{l}(x)=0\}\,,
\end{equation}
where $p_i,q_j\in\R[x]$.
We call $p_i$ (resp. $q_j$) inequality (resp. equality) polynomial defining $W$.
A set is called semi-algebraic, if it is the union of a finite number of elementary semi-algebraic sets.
Denote by $\overline{A}$ the closure of a set $A\subset\R^n$ in the usual topology on $\R^n$.

Known as a central result of real algebraic geometry, Tarski--Seidenberg's theorem \cite{tarski1951decision,seidenberg1954new} says that the projection of a semi-algebraic subset of $\R^{n}$ onto the space spanned by the first $n-1$ coordinates is  semi-algebraic. 
%Consequently, the image of a semi-algebraic set under a polynomial mapping is semi-algebraic.
Using a finite number of arithmetic operations, we can write down explicitly the equations and inequalities defining this projection.
Therefore, algorithms that allow us to do so are fundamental and widely used in various applications of real algebraic geometry.

\paragraph{Contribution.}
In this paper, we provide an algorithm to symbolically compute polynomials defining the closure of the image of a semi-algebraic set $S$ under a polynomial $f$.
Moreover, we obtain explicitly $\overline{f(S)}$ as a finite union of closed semi-algebraic sets of the form
\begin{equation}\label{eq:cl.semialg.set}
\{t\in\R\,:\,g_1(t)\ge 0,\dots,g_v(t)\ge 0\,,\,h_1(t)=\dots=h_k(t)=0\}\,,
\end{equation}
where $g_i,h_j$ are univariate polynomials.
Our method consists of the following three steps:
\begin{enumerate}
\item By Tarski--Seidenberg's theorem, the image $f(S)$ is a semi-algebraic set. Moreover, we obtain its (in)equality polynomials by using Basu's algorithm in \cite[Section 3]{basu1999new}. 
\item 
We then convert these polynomial descriptions to new ones defining elementary semi-algebraic sets $B_1,\dots,B_r$ whose union is $f(S)$ such that the following conditions hold:
\begin{enumerate}
\item Each inequality polynomial defining $B_j$ has only simple roots.
\item Any inequality polynomial defining $B_j$ has no root in the variety determined by the equality polynomials defining $B_j$.
\item Any couple of inequality polynomials determining each $B_j$ has no common root. 
\end{enumerate}
\item Replacing strict inequalities defining $B_j$ with non-strict ones, we obtain a new semi-algebraic set $\widetilde B_j$.
We guarantee that the closure $\overline{f(S)}$ is identical to the union of $\widetilde B_j$s. 
\end{enumerate}
The main tools utilized in the last two steps are the fundamental theorem of algebra and the greatest common divisor.

Let $d$ be an upper bound on the degrees of $f$ and polynomials defining $S$. 
Our method has a complexity of $O(d^{O(n)})$ to return polynomials of degrees at most $d^{O(n)}$ that define $\overline{f(S)}$.
This complexity follows from the complexity of Tarsi--Seidenberg's theorem analyzed by Basu in \cite[Theorem 1]{basu1999new}.

%Consequently, we obtain a new type of sums-of-squares representation for every polynomial $h$ nonnegative on semi-algebraic set $S$ in the case where the image $h(S)$ is bounded and $h$ possibly has no real zero on $S$.

(Another possible way is to rely on Basu's method in \cite[Section 2.1, Example]{basu1999new} for computing the closure of a semi-algebraic set to find the semi-algebraic description of $\overline{f(S)}$. 
%However, we indicate an example in Remark \ref{rem:basu} that his method is incorrect in general.
However, we are also not sure if this method allows us to obtain $\overline{f(S)}$ as a finite union of closed semi-algebraic sets of the form \eqref{eq:cl.semialg.set}, on which we can build some Nichtnegativstellens\"atze for univariate polynomials.)
\paragraph{Motivation.} One of the applications of our method is to address the attainability issue in polynomial optimization.
More explicitly, consider 
\begin{equation}\label{eq:pop}
f^\star=\inf_{x\in S}f(x)\,,
\end{equation}
where $f$ is a polynomial in $\R[x]$ and $S$ is a semi-algebraic subset of $\R^n$. 
Assume that the optimum value $f^\star$ is finite, i.e., $f^\star\in\R$.

It is well-known that when $S$ is a compact semi-algebraic set of the form
\begin{equation}\label{eq:cl.semialg.set.closed}
\{x\in\R^n\,:\,p_1(x)\ge 0,\dots,p_m(x)\ge 0\,,\,q_1(x)=\dots=q_{l}(x)=0\}\,,
\end{equation}
the optimal value  $f^\star$ is approximated from below as closely as desired by the sum-of-squares strengthenings introduced by Lasserre in \cite{lasserre2001global}.
If $f$ attains $f^\star$ on non-compact $S$ of form \eqref{eq:cl.semialg.set.closed}, we can approximately compute $f^\star$ using Mai--Lasserre--Magron's in \cite{mai2021positivity}.
Moreover, tools that allow us to compute $f^\star$ exactly in Demmel--Nie--Powers' and Mai's work  \cite{demmel2007representations,mai2022exact} require the attainability of $f^\star$.

However, the problem \eqref{eq:pop} possibly has no optimal solution (i.e., $f^\star$ is not attained).
For instance, if $f=x_1$ and $S=\{x\in\R^2\,:\,x_1x_2^2=1\}$, then $f^\star=0$ is not attained by $f$ on $S$.
 
We now formulate \eqref{eq:pop} as $f^\star=\inf f(S)$. 
A simple idea to make $f^\star$ attained is replacing the image $f(S)$ with its closure. 
This gives
\begin{equation}\label{eq:attain.image}
f^\star=\min \overline{f(S)}
\end{equation}
since $f^\star\in \overline{f(S)}$.
The challenge is to transform \eqref{eq:attain.image} back to a polynomial optimization problem with attained optimum value $f^\star$.
In other words,  we find a polynomial $g$ and a semi-algebraic set $Q$ such that 
\begin{equation}\label{eq:pop.attained}
f^\star=\min_{y\in Q}g(y)\,.
\end{equation}

Our method allows us to obtain $Q=\overline{f(S)}$ and $g$ as the identity polynomial in a single variable, i.e., $g(t)=t$.
In principle, $f(S)$ is a semi-algebraic subset of $\R$ and therefore is the union of finitely many intervals $I_1,\dots,I_s$ in $\R$.
Here $I_j$s are of the following forms: $(a,b)$, $[a,b)$, $(a,b]$, $[a,b]$, $(a,\infty)$, $[a,\infty)$, $(-\infty,b]$, $(-\infty,b]$, $[a,a]=\{a\}$ with $-\infty<a< b<\infty$.

The closure $\overline{f(S)}$ is then the union of $\overline{I_1},\dots,\overline{I_s}$, where each $\overline{I_j}$ contains $I_j$ and its endpoints.
Note that the endpoints of $I_j$s are the real roots of (in)equality polynomials defining $f(S)$.
It leads to using numerical methods to find the real roots of a system of univariate polynomials.
In contrast, our method relies only on symbolic computation to obtain the equations and inequalities determining $\overline{f(S)}$.

It remains to solve the problem \eqref{eq:pop.attained},
where $Q=\overline{f(S)}$ is a semi-algebraic subset of $\R$ and $g$ is the identity polynomial.
In this case, we can exactly compute the optimum value $f^\star$ by using the sum-of-squares strengthenings under the boundedness assumption of $f(S)$.
\paragraph{Previous works.}
In \cite{mai2022symbolic}, the author provides a symbolic algorithm to compute polynomials defining the image $f(S)$ when $f(S)$ has finitely many points in $\R$ and $S$ is of form \eqref{eq:cl.semialg.set.closed}. 
In this case, $f(S)=\overline{f(S)}$, and $f(S)$ is also identical to its Zariski closure, the smallest variety  containing $f(S)$.
Here a variety is an elementary semi-algebraic set defined only by equations.
The method in \cite{mai2022symbolic} depends on the computations of real radical generators and Groebner bases.

Regarding the attainability issue for the polynomial optimization problem \eqref{eq:pop}, we refer the readers to \cite{schweighofer2006global,
ha2009solving,dinh2014frank,pham2019optimality}.
In \cite{schweighofer2006global}, Schweighofer deals with the case where $S=\R^n$ and $f^\star$ is not attained by $f$.
He replaces the gradient ideals used in Nie--Demmel--Sturmfels' method in \cite{nie2006minimizing} with the gradient tentacles to obtain appropriate sum-of-squares strengthenings for the problem \eqref{eq:pop}.
H\`a and Pham handle the constraint case of $S$ of form \eqref{eq:cl.semialg.set.closed} in \cite{ha2009solving}.
They propose the truncated tangency variety $\Lambda$ for $f$ and $S$ with property saying that $f$ has infimum value $f^\star$ on $\Lambda$ under regularity assumption even if $f$ does not attain on $S$. 
It also allows them to obtain sum-of-squares strengthenings for the problem \eqref{eq:pop}.
Dinh, Ha, and Pham provide in \cite{dinh2014frank} a Frank--Wolfe type theorem for $f^\star$ bounded from below on $S=\{x\in\R^n\,:\,p_1(x)\ge 0,\dots, p_m(x)\ge 0\}$.
It says that if the polynomial map $x\mapsto (f(x),p_1(x),\dots,p_m(x))$ is convenient and non-degenerate at infinity, then $f$ attains its infimum $f^\star$ on $S$.
In \cite{pham2019optimality}, Pham derives versions at infinity of the Fitz John and  Karush--Kuhn--Tucker conditions to deal with the case where $f^\star$ is not attained by $f$ on $S$ of form \eqref{eq:cl.semialg.set.closed}.
He also indicates the existence of polynomial $g$ and semi-algebraic set $Q$ such that \eqref{eq:pop.attained} holds under the non-degenerate assumption at infinity.

For comparison purposes, our method in this paper does not require any assumption on $f$ and $S$ except the finiteness of $f^\star$.
Besides, we convert the original problem \eqref{eq:pop} in $n$-dimensional space to equivalent form \eqref{eq:pop.attained} in one-dimensional space.
The methods in \cite{schweighofer2006global,
ha2009solving,pham2019optimality,dinh2014frank} preserve the  dimension in obtaining equivalent forms.
Moreover, their transformations to equivalent forms are much simpler than ours as we use a lot of arithmetic operations on $f$ and polynomials defining $S$.
\paragraph{Organization.}
We organize the paper as follows: 
Section \ref{sec:Preliminaries} presents some preliminaries and lemmas needed to prove our main results.
Section \ref{sec:prod.linear.factors.odd} is to provide an algorithm that allows us to obtain the product of linear factors of a polynomial occurring odd times.
Section \ref{sec:replace.strict.ineq} is to prove that replacing strict inequalities defining a ``nice'' elementary semi-algebraic set with non-strict ones allows us to obtain the closure of this set.
Section \ref{sec:main.alg.app.poly.opt} is to present the main algorithm that enables us to obtain semi-algebraic descriptions of the closures of the images of semi-algebraic sets under polynomials.
%In Section \ref{sec:representation}, we state  Nichtnegativstellens\"atze for every polynomial $h$ nonnegative on a semi-algebraic set $S$ in the case where the image of $S$ under $h$ is bounded and $h$ possibly has no real zero on $S$.
\section{Preliminaries}
\label{sec:Preliminaries}

\subsection{Projections of semi-algebraic sets and quantifier elimination}
We restate Tarski--Seidenberg's theorem in the following lemma:
\begin{lemma}\label{lem:Tarski--Seidenberg}
Let $S$ be a semi-algebraic subset of $\R^n$ and $\pi:\R^{n}\to \R^{n-1}$ be the projector onto the space spanned by the first $n-1$ coordinates. 
Then the projection $\pi(S)$ is semi-algebraic. 
\end{lemma}

A first-order formula in the language of real closed fields with free variables $x =(x_1,\dots,x_n)$ is obtained as follows recursively:
\begin{enumerate}
\item If $f \in \R[x]$, $n \ge 1$, then
$f=0$ and $f>0$ are first-order formulas and $\{x\in\R^n\,:\,f(x)=0\}$ and $\{x\in\R^n\,:\, f(x) > 0\}$ are respectively the subsets of $\R^n$ such that the formulas $f = 0$ and $f>0$ hold.
\item If $\varphi$ and $\psi$ are first-order formulas, then $\varphi \wedge \psi$ (conjunction), $\varphi \lor \psi$ (disjunction)
and $\lnot\varphi$ (negation) are also first-order formulas.
\item If $\varphi$ is a first-order formula and $x$ is a variable ranging over $\R$, then $\exists x\varphi$ and $\forall x\varphi$ are first-order formulas.
\end{enumerate}
The formulas obtained by using only rules 1 and 2 are called quantifier-free
formulas.

It follows from definitions that a subset $S \subset \R^n$ is semi-algebraic if and only if
there exists a quantifier-free formula $\varphi$ such that
\begin{equation}\label{eq:semial.quantifier}
S=\{x\in\R^n\,:\,\varphi(x)\}\,.
\end{equation}

Eliminating the quantifier from $\exists x_{n}\varphi$, is the same as computing a quantifier-free description of the image 
\begin{equation}
\pi(S)=\{x'\in \R^{n-1}\,:\,\exists x_{n}\varphi(x',x_n)\}\,,
\end{equation}
where $x'=(x_1,\dots,x_{n-1})$ and $\pi$ is the projector defined as in Lemma \ref{lem:Tarski--Seidenberg}.
Tarski--Seidenberg's theorem (Lemma \ref{lem:Tarski--Seidenberg}) allows us to replace $\exists x_{n}\varphi(x',x_{n})$ by some quantifier-free formula with vector of variables $x'$. 
%By induction, it follows that any quantified formula may be replaced by a quantifier-free one.

The following example is to illustrate the equivalence between eliminating a quantifier and computing semi-algebraic description of the projection of a semi-algebraic set:
\begin{example}
Since $\exists t(at^2+bt+c = 0)$
is equivalent to the quantifier-free formula $(a \ne 0 \wedge b^2 - 4ac \ge 0) \lor (a = 0 \wedge b \ne
0) \lor (a = 0 \wedge b = 0 \wedge c = 0)$, the projection of semi-algebraic set 
\begin{equation}
\{(a,b,c,t)\in\R^4\,:\,at^2+bt+c = 0\}
\end{equation}
onto the space spanned by the first three coordinates is semi-algebraic set
\begin{equation}
\begin{array}{rl}
&\{(a,b,c)\in\R^3\,:\,a \ne 0,b^2 - 4ac \ge 0\}\\
\cup &\{(a,b,c)\in\R^3\,:\,a = 0 , b \ne
0\}\\
\cup& \{(a,b,c)\in\R^3\,:\,a = b = c = 0\}\,.
\end{array}
\end{equation}
\end{example}
Basu suggests his quantifier elimination algorithm in \cite[Section 3]{basu1999new}.
We use his method in the first step of our main algorithm stated later.
\begin{remark}\label{rem:complex.Tarski}
Heintz, Roy, and  Solern\'o analyze the complexity of Tarski--Seidenberg's theorem in \cite{heintz1990complexite}.
An improvement is given by Basu in \cite{basu1999new} using terms of quantifier elimination.
Let $S\subset \R^n$ be a semi-algebraic set determined by $s$ polynomials of degree at most $d$.
Basu's method requires $s^{2}d^{nO(1)}$ arithmetic operations to find polynomials of degrees upper bounded by $d^{O(1)}$ that define the projection of $S$ onto the space spanned by the first $n-1$ coordinates.
\end{remark}

\subsection{Polynomial images of semi-algebraic sets}
A mapping $\varphi:\R^n\to \R^d$ is called semi-algebraic if its graph $\{(x,\varphi(x))\,:\,x\in\R^n\}$ is a semi-algebraic subset of $\R^{n+d}$.
The following lemma follows from Tarski--Seidenberg's theorem (Lemma \ref{lem:Tarski--Seidenberg}):
\begin{lemma}\label{lem:consequence.Tarski--Seidenberg}
Let $S$ be a semi-algebraic subset of $\R^n$ and $\varphi:\R^n\to \R^d$ be a semi-algebraic mapping. 
Then the image $\varphi(S)$ is semi-algebraic. 
\end{lemma}
\begin{remark}\label{rem:alg.consequence.Tarski--Seidenberg}
%The equations and inequalities defining the image $\varphi(S)$ in Lemma \ref{lem:Tarski--Seidenberg} can be written down explicitly, using a finite number of arithmetic operations on the equations and inequalities which define the set $S$ and the graph of the mapping $\varphi$.
Korkina and Kushnirenko provide in \cite{korkina1985another} an algorithm to find the equations and inequalities defining the image $\varphi(S)$ (see also Coste's book \cite[Sections 1 and 2]{coste2000introduction}).
\end{remark}

The following lemma is a trivial consequence of Lemma \ref{lem:consequence.Tarski--Seidenberg} and Remark \ref{rem:alg.consequence.Tarski--Seidenberg}:
\begin{lemma}\label{lem:polyimage.semi}
Let $f$ be a polynomial and $S$ be a semi-algebraic set.
Then the image $f(S)$ is a semi-algebraic subset of $\R$.
Moreover, there is a symbolic algorithm that produces the equations and inequalities defining $f(S)$.
\end{lemma}

Next, we analyze the complexity of the algorithm mentioned in Lemma \ref{lem:polyimage.semi}:
\begin{remark}\label{rem:complex.polyimage.semi}
Let $S\subset \R^n$ be a semi-algebraic set determined by $s$ polynomials of degree at most $d$.
Let $f$ be a polynomial of degree at most $d$.
It is not hard to prove that $f(S)$ is the projection of semi-algebraic set 
\begin{equation}
\{(x,x_{n+1})\in\R^{n+1}\,:\,x_{n+1}=f(x)\,,\,x\in S\}
\end{equation}
onto the space spanned by the last coordinate.
It is equivalent to eliminating a block of
quantifiers of length $n$.
Thanks to the complexity given in \cite[Theorem 1]{basu1999new}, it requires $s^{n+1}d^{O(n)}$ arithmetic operations to find the polynomials of degrees upper bounded by $d^{O(n)}$ defining $f(S)$.
\end{remark}

\subsection{Closures of semi-algebraic sets in high-dimensional spaces}
\label{sec:descrip.cl}
Set $\|x\|_2^2:=x_1^2+\dots+x_n^2$. 
We recall some properties of semi-algebraic sets in the following lemma:
\begin{lemma}\label{lem:properties.semi-algebraic.set}
The following statements hold:
\begin{enumerate}
\item With $A$ and $B$ being semi-algebraic subsets of $\R^n$, the sets
$A \cup B$, $A \cap B$, and $\R^n \backslash A$ are also semi-algebraic.
\item The closure of a semi-algebraic set is semi-algebraic.
\item A semi-algebraic subset of $\R$ is a finite union of intervals and points in $\R$.
\end{enumerate}
\end{lemma}
\begin{proof}
The first and third statements are given in \cite[Proposition 1.1.1]{pham2016genericity} and \cite[Example 1.1.1]{pham2016genericity}, respectively. The second one is proved by Coste in \cite[Corollary 2.5]{coste2000introduction} as follows:
Let $S$ be a semi-algebraic set defined as in \eqref{eq:semial.quantifier}.
The closure of $S$ is
\begin{equation}\label{eq:cl.S.quantifier}
\overline S=\{x\in\R^n\,:\,(\forall\varepsilon)(\varepsilon>0) (\exists y) (\|x-y\|_2^2<\varepsilon^2)\wedge \varphi(y)\}\,.
\end{equation}
and can be written as
\begin{equation}
\overline{S} = \R^n \backslash (\pi_1 (\{(x, \varepsilon) \in \R^n\times \R \,,\, \varepsilon > 0\} \backslash \pi_2(B))) ,
\end{equation}
where
\begin{equation}
B =\{(x, \varepsilon, y) \in \R^n \times \R \times \R^n \,:\, (\|x-y\|_2^2<\varepsilon^2)\wedge \varphi(y)\}\,,
\end{equation}
$\pi_1(x, \varepsilon) = x$ and $\pi_2(x, \varepsilon, y)=(x, \varepsilon)$. Then observe that $B$ is semi-algebraic.
Hence $\overline{S}$ is semi-algebraic thanks to Lemma \ref{lem:Tarski--Seidenberg} and the first statement. 
\end{proof}

%To the best of our knowledge, there is still no correct algorithm so far that allows us to find the exact equations and inequalities defining $\overline{S}$ in Proposition \ref{prop:cl}.
\begin{remark}\label{rem:basu}
In \cite{magron2015semidefinite}, Magron, Henrion, and Lasserre approximate the closure $\overline{S}$ in Lemma  \ref{lem:properties.semi-algebraic.set} using semidefinite programming.
Basu develops in \cite[Section 2.1, Example]{basu1999new} a symbolic method to compute the closure of a semi-algebraic set $S$ defined as in \eqref{eq:semial.quantifier}, where $\varphi(x)$ is a quantifier-free first order formula involving $s$ polynomials with degrees bounded by $d$.
It is based on quantifier elimination and has complexity $s^{2(n+1)}d^{O(n^2)}$.
His idea is to eliminate two blocks of
quantifiers corresponding to $\varepsilon$ and $y$ for the description of $\overline S$ in \eqref{eq:cl.S.quantifier}. 
\end{remark}

\begin{remark}
With $\varphi$ and $S$ being as in Lemma \ref{lem:Tarski--Seidenberg}, the closure $\overline{\varphi(S)}$ is semi-algebraic thanks to the second statement of Lemma \ref{lem:properties.semi-algebraic.set}.
By the third statement of Lemma \ref{lem:properties.semi-algebraic.set}, the image $f(S)$ in Lemma \ref{lem:polyimage.semi} is a finite union of points and  intervals of $\R$.
It implies that the closure $\overline{f(S)}$ is semi-algebraic and therefore is a finite union of closed intervals and points of $\R$. %and therefore is semi-algebraic. 
\end{remark}
\if{
\begin{lemma}
Let $S$ be a semi-algebraic set of $\R^{k+n}$ defined by s polynomials in $k+n$ variables such that the sum of their degrees is less than or equal to $D$.
There is an algorithm of sequential complexity $D^{(k+n)^{\mathcal{O}(1)}}$ and parallel complexity $((k + n)\log D)^{\mathcal{O}(1)}$ constructing:
\begin{enumerate}
\item a partition of $\R^k$ into semi-algebraic sets $T_m$;
\item for each $m$ a finite family $(\xi_{m,j})_{j=1,\dots,l_m}$, of continuous semi-algebraic functions from $T_m$ to $\R^n$ such that the graph of $\xi_{m,j}$ is contained in $S$ and such that for all $y\in T_m$ the graphs $\xi_{m,j}$ intersect each (semi-algebraically) connected component of $S \cap (\{y\}\times\R^n)$.
\end{enumerate}
\end{lemma}
}\fi

\if{
Let $B(a,r)$ stand for the open ball in $\R^n$ centered at $a\in\R^n$ with radius $r>0$.
The following proposition says that the closure of $S$ is an infinite intersection of open sets containing $S$:
\begin{proposition}\label{prop:cl}
Let $S$ defined as in \eqref{eq:semial.quantifier} with $\varphi(x)$ being a quantifier-free first order formula.
Then
\begin{equation}
\begin{array}{rl}
\overline S&=\{y\in\R^n\,:\,(\forall\varepsilon >0)\,,\, (\exists x) (\|x-y\|_2^2<\varepsilon^2)\wedge \varphi(x)\}\\[5pt]
&=\bigcap\limits_{\varepsilon>0}\{y\in \R^n\,:\,(\exists x) (\|x-y\|_2^2<\varepsilon^2)\wedge \varphi(x)\}\,.
\end{array}
\end{equation}
\end{proposition}
\begin{proof}
Set $A=\{y\in\R^n\,:\,(\forall\varepsilon >0)\,,\, (\exists x) (\|x-y\|_2^2<\varepsilon^2)\wedge \varphi(x)\}$.
By definition, $y\in \overline{S}$ if and only if for every $\varepsilon>0$, $B(y,\varepsilon)\cap S\ne \emptyset$.
It implies $\overline S=A$, yielding the result.
\end{proof}
}\fi

\subsection{Fundamental theorem of algebra and greatest common divisor}
Let $\R[t]$ be the ring of polynomials in single-variable $t$.
Let $\deg(p)$ stand for the degree of a polynomial $p\in\R[t]$.

We restate the fundamental theorem of algebra in the following lemma:
\begin{lemma}\label{lem:factor.poly}
Every polynomial $p\in\R[t]$ of positive degree can be decomposed uniquely as:
\begin{equation}\label{eq:factor.poly}
p=c\times\prod\limits_{i=1}^{s} (t-a_i)^{m_i}\,,
\end{equation}
where $c\in\R$, $a_i\in \C$, and $s,m_i\in\N$ such that $c\ne 0$, $m_i>0$ and $a_1,\dots,a_s$ are disjoint.
Moreover, for every $i\in\{1,\dots,s\}$ satisfying $a_i\in\C\backslash \R$, there is $j\in \{1,\dots,s\}\backslash \{i\}$ such that $m_j=m_i$, and $a_j$ is the conjugate of $a_i$.
\end{lemma}
We call $a_i$ in Lemma \ref{eq:factor.poly} a (complex) root of $p$ with multiplicity $m_i$.
If $m_i>1$ (resp. $m_i=1$), $a_i$ is call multiple (resp. simple) root of $p$.

Let $\gcd(p_1,\dots,p_m)$ stand for the greatest common divisor of polynomials $p_1,\dots,p_m$ in $\R[t]$.
To find $\gcd(p_1,\dots,p_m)$, we use the well-known polynomial Euclidean algorithm (see, e.g., Basu--Pollack--Roy's book \cite{basu2003algorithms}).
\begin{remark}\label{rem:complex.gcd}
Belhaj--Kahla's algorithm in \cite{belhaj2013complexity} to compute $\gcd(p_1,p_2)$ via Hankel matrices has a cost of $O(d^2)$ arithmetic operations, where $d=\max\{\deg(p_1),\deg(p_2)\}$.
Since $\gcd(p_1, p_2, \dots , p_m) = \gcd( p_1, \gcd(p_2, \dots , p_m))$, we obtain the number of arithmetic operations to compute $\gcd(p_1,\dots , p_m)$ as $(m-1)\times O(d^2)$, where $d=\max\limits_{j=1,\dots,m}\deg(p_j)$.
\end{remark}

Denote by $p'$ the derivative of a polynomial $p\in\R[t]$.
The following two lemmas present the basic properties of the roots of a polynomial and its derivative:
\begin{lemma}\label{lem:mul.root}
Let $a$ be a root of a polynomial $p\in\R[t]$.
Then $a$ is a multiple root of $p$ if and only if $a$ is also a root of $p'$.
\end{lemma}
\begin{proof}
Since $a$ is a root of $p$, there exists a polynomial $q$ such that $p=(t - a)q$.
Using the product rule of derivatives, we know that
$p' = q + (t-a)q'$.
But then $(t-a)$ only divides $p'$ if and only if it also divides $q$.
\end{proof}
\begin{lemma}\label{lem:simp.root}
If a polynomial $p\in\R[t]$ has only simple roots.
Then $p,p'$ are relatively prime, i.e., $\gcd(p,p')=1$. 
\end{lemma}
\begin{proof}
Assume that $p,p'$ are not relatively prime.
Then, they have a common irreducible factor of degree $1$.
Since by definition, two polynomials are relatively prime if their only common factor is of degree $0$.
Then there exists a polynomial of the form $t-a$ that divides both $p$ and $p'$.
Then from Lemma \ref{lem:mul.root}, $a$ is a multiple root of $p$.
But this is impossible.
\end{proof}

We recall the radical of a polynomial in the following lemma:
\begin{lemma}\label{lem:radical}
Let $p$ be polynomial in $\R[t]$ with decomposition \eqref{eq:factor.poly}.
Then
\begin{equation}\label{eq:radical1}
\arraycolsep=1.4pt\def\arraystretch{.5}
\gcd(p,p')=\prod_{\begin{array}{cc}
\scriptstyle i=1 \\
 \scriptstyle  m_i>1
\end{array}}^{s} (t-a_i)^{m_i-1}\,,
\end{equation}
\begin{equation}\label{eq:radical2}
\frac{p}{\gcd(p,p')}=\prod_{i=1}^{s} (t-a_i)
\end{equation}
are polynomials in $\R[t]$.
\end{lemma}
\begin{proof}
Set  
\begin{equation}
\arraycolsep=1.4pt\def\arraystretch{.5}
h=\prod_{i=1}^{s} (t-a_i)\qquad\text{and}\qquad w=\prod_{\begin{array}{cc}
\scriptstyle i=1 \\
 \scriptstyle  m_i>1
\end{array}}^{s} (t-a_i)^{m_i-1}
\end{equation}
 Then $p=wh$ and
\begin{equation}
\arraycolsep=1.4pt\def\arraystretch{.5}
h'=\sum_{j=1}^s \prod_{\begin{array}{cc}
\scriptstyle i=1 \\
 \scriptstyle  i\ne j
\end{array}}^{s} (t-a_i)\,.
\end{equation}
Lemma \ref{lem:simp.root} says that $\gcd(h,h')=1$ since $h$ has only simple roots. 
Simple computation gives
\begin{equation}
\arraycolsep=1.4pt\def\arraystretch{.5}
p'=\sum_{j=1}^s m_j(t-a_j)^{m_j-1}\prod_{\begin{array}{cc}
\scriptstyle i=1 \\
 \scriptstyle  i\ne j
\end{array}}^{s} (t-a_i)^{m_i}=wh'\,.
\end{equation}
From this and $\gcd(h,h')=1$, we obtain
$\gcd(p,p')=w\gcd(h,h')=w$,
which yields \eqref{eq:radical1} and \eqref{eq:radical2}.
In addition, $w$ and $h$ are polynomials thanks to the second statement of Lemma \ref{lem:factor.poly}.
Hence the result follows.
\end{proof}

Lemma \ref{lem:radical} says that we can symbolically compute the product of linear factors of a polynomial.
The following lemma shows how to get the product of linear factors of a polynomial occurring once:
\begin{lemma}\label{lem:mul.one}
Let $p$ be polynomial in $\R[t]$ with decomposition \eqref{eq:factor.poly}.
Set $q:=\gcd(p,p')$.
Then
\begin{equation}\label{eq:prod.mul.one}
\arraycolsep=1.4pt\def\arraystretch{.5}
\frac{p\times\gcd(q,q')}{q^2}=\prod_{\begin{array}{cc}
\scriptstyle i=1 \\
 \scriptstyle  m_i=1
\end{array}}^{s} (t-a_i)
\end{equation}
is a polynomial in $\R[t]$.
\end{lemma}
\begin{proof}
Applying Lemma \ref{lem:radical}, we get 
\begin{equation}
\arraycolsep=1.4pt\def\arraystretch{.5}
\frac{q}{\gcd(q,q')}=\prod_{\begin{array}{cc}
\scriptstyle i=1 \\
 \scriptstyle  m_i>1
\end{array}}^{s} (t-a_i)\,.
\end{equation}
It implies that
\begin{equation}
\arraycolsep=1.4pt\def\arraystretch{.5}
\frac{p}{\gcd(p,p')}\times \frac{\gcd(q,q')}{q}=\prod_{i=1}^{s} (t-a_i)\div \prod_{\begin{array}{cc}
\scriptstyle j=1 \\
 \scriptstyle  m_j>1
\end{array}}^{s} (t-a_j)=\prod_{\begin{array}{cc}
\scriptstyle i=1 \\
 \scriptstyle  m_i=1
\end{array}}^{s} (t-a_i)\,.
\end{equation}
Hence the result follows.
\end{proof}

\subsection{Closures of semi-algebraic sets in one-dimensional spaces}
The following lemma describes the properties of the closures of the intersection and union of several sets:
\begin{lemma}\label{lem:intersec.closure}
Let $A_1,\dots,A_m$ be a finite collection of subsets of $\R^n$.
Then the following two statements hold: 
\begin{enumerate}
\item $\overline{\bigcap_{i=1}^m A_i}\subset \bigcap_{i=1}^m \overline{A_i}$. 
\item $\overline{\bigcup_{i=1}^m A_i}= \bigcup_{i=1}^m \overline{A_i}$.
\end{enumerate}
\end{lemma}
\begin{proof}
It is not hard to prove the second statement. 
Let us prove the first one.
For $j=1,\dots,m$, we get $\bigcap_{i=1}^m A_i\subset A_j$, which gives $\overline{\bigcap_{i=1}^m A_i}\subset \overline{A_j}$.
Hence the result follows.
\end{proof}

We say that a polynomial $p\in\R[t]$ changes sign at its real root $a$, if 
\begin{equation}
\lim\limits_{t\to a^+}\frac{1}{p(t)}\ne \lim\limits_{t\to a^-}\frac{1}{p(t)}\,.
\end{equation}
Note that these limits have only values $\pm\infty$.
The following lemma shows that the sign change of a polynomial depends on the multiplicities of real roots:
\begin{lemma}\label{lem:chage.sign}
Let $p$ be a polynomial in $\R[t]$ of positive degree.
Then $p$ changes (resp. does not change) sign at its real roots with odd (resp. even) multiplicities.
\end{lemma}
\begin{proof}
Let $a$ be a real root of $p$ with multiplicity $m$. Then $p=q(t-a)^m$, where $q$ is some polynomial in $\R[t]$ that is not zero at $a$. 
Consider the following two cases:
\begin{itemize}
\item Case 1: $m$ is odd. Suppose that  $q(a)>0$. By continuity $q$ is positive on an open interval containing $a$ implying that $\lim\limits_{t\to a^+}\frac{1}{p(t)}=\infty\ne -\infty=\lim\limits_{t\to a^-}\frac{1}{p(t)}$. Thus $p(t)$ changes sign at $a$. The same will be true if  $q(a)<0$. 
\item Case 2: $m$ is even.  If  $q(a)>0$, $q$ is positive near $a$ so $p$ is also positive on a deleted neighborhood of $a$.
It implies that $\lim\limits_{t\to a^+}\frac{1}{p(t)}=\infty=\lim\limits_{t\to a^-}\frac{1}{p(t)}$. Thus $p(t)$ does not change sign at $a$.
The same will be true if  $q(a)<0$. 
\end{itemize}
\end{proof}
A semi-algebraic description of an open semi-algebraic set defined only by a single inequality is present in the following lemma: 
\begin{lemma}\label{lem:odd.mul}
Let $p$ be a polynomial in $\R[t]$ of positive degree.
Suppose that each real root of $p$ has odd multiplicity.
Then $\{p\ge 0\}=\overline{\{p>0\}}$.
\end{lemma}
\begin{proof}
Let $a_1,\dots,a_s$ be the real roots of $p$ such that $a_1<\dots<a_s$.
By assumption, these roots have odd multiplicities.
Since $p$ changes sign at each $a_i$ (by Lemma \ref{lem:chage.sign}), one of the following two cases occurs:
\begin{itemize}
\item Case 1: The leading coefficient of $p$ is negative, i.e., $\lc(p)<0$. Then
\begin{equation}
\{p>0\}=\begin{cases}
(-\infty,a_1)\cup(a_2,a_3)\cup \dots\cup (a_s,\infty)&\text{if $s$ is even}\,,\\
(-\infty,a_1)\cup(a_2,a_3)\cup \dots\cup (a_{s-1},a_s)&\text{otherwise}\,.\\
\end{cases}
\end{equation}
\item Case 2: $\lc(p)>0$. Then
\begin{equation}
\{p>0\}=\begin{cases}
(a_1,a_2)\cup(a_3,a_4)\cup \dots\cup (a_s,\infty)&\text{if $s$ is odd}\,,\\
(-\infty,a_1)\cup(a_2,a_3)\cup \dots\cup (a_s,\infty)&\text{otherwise}\,.\\
\end{cases}
\end{equation}
\end{itemize}
Thus $\{p\ge 0\}=\{p>0\}\cup\{a_1,\dots,a_s\}=\overline{\{p>0\}}$ yields the result.
\end{proof} 
Denote by $\inter(A)$ the interior of a set $A\subset\R^n$.
Let $\partial(A)$ stand for the boundary of $A\in\R^n$, i.e., $\partial(A)=\overline{A}\backslash \inter(A)$.
The following lemma shows a semi-algebraic description of the closure of an open semi-algebraic set:
\begin{lemma}\label{lem:closure.strict}
Let $p_1,\dots,p_m$ be polynomials in $\R[t]$ of positive degrees.
Suppose that the following conditions hold:
\begin{enumerate}
\item For $i=1,\dots,m$, each real root of $p_i$ has odd multiplicity.
\item Any couple of $p_1,\dots,p_m$ has no common real root.
\end{enumerate}
Then 
\begin{equation}
\overline{\bigcap_{i=1}^m \{p_i>0\}}=\bigcap_{i=1}^m \{p_i\ge 0\}\,.
\end{equation}
\end{lemma}
\begin{proof}
Set 
\begin{equation}\label{eq:set.stric.nonstrict}
A=\bigcap_{i=1}^m \{p_i\ge 0\}\qquad\text{and}\qquad B=\bigcap_{i=1}^m \{p_i>0\}\,.
\end{equation}
We claim that $\overline{B}\subset A$.
By assumption, Lemma \ref{lem:odd.mul} yields $\overline{\{ p_i>0\}}=\{ p_i\ge 0\}$.
Form this and the first statement of Lemma \ref{lem:intersec.closure}, we obtain
\begin{equation}
\overline{B}\subset \bigcap_{i=1}^m \overline{\{ p_i>0\}}=\bigcap_{i=1}^m \{p_i\ge 0\}=A\,.
\end{equation}
It is sufficient to prove that $A\subset \overline{B}$.
Let $a\in A$. If $a\in \inter(A)$, then $a\in B$ implies that $a\in\overline{B}$.
Assume that $a\in\partial(A)$.
Then $a$ is a root of $p_j$ for some $j\in\{1,\dots,m\}$. 
By assumption, $a$ is a root of $p_j$ with odd multiplicity.
By Lemma \ref{lem:chage.sign}, $p_j$ changes sign at $a$.
Thus there is a sequence $(a_k)_{k\in\N}\subset \{p_j>0\}$ such that $\lim\limits_{k\to\infty}a_k=a$.
By assumption, $a$ is not a root of $p_i$ for any $i\in\{1,\dots,m\}\backslash \{j\}$.
Since $a\in A$, it implies that $a$ is in the open set
\begin{equation}
\arraycolsep=1.4pt\def\arraystretch{.5}
U=\bigcap_{\begin{array}{cc}
\scriptstyle i=1 \\
 \scriptstyle i\ne j
\end{array}}^{m} \{p_i>0\}\,.
\end{equation}
Since the sequence $(a_k)_{k\in\N}\subset \{p_j>0\}$ converges to $a$, there exists a subsequence 
$(a_{k_r})_{r\in\N}\subset B=U\cap \{p_j>0\}$ that converges to $a$.
It implies that $a\in \overline{B}$.
Hence the result follows.
\end{proof}
In the following lemma, we present the property of the intersection of the closure of an open semi-algebraic set with a variety:
\begin{lemma}\label{lem:closure.strict.intersec}
Let $p_1,\dots,p_m$ be polynomials in $\R[t]$ of positive degrees.
Let $V$ be a finite subset of $\R$ such that any $p_i$ has no root in $V$.
Then 
\begin{equation}
\overline{\bigcap_{i=1}^m \{p_i>0\}}\cap V=\overline{\bigcap_{i=1}^m \{p_i>0\}\cap V}\,.
\end{equation}
\end{lemma}
\begin{proof}
Set $A$ as in \eqref{eq:set.stric.nonstrict}.
The first statement of Lemma \ref{lem:intersec.closure} yields $\overline{A\cap V}\subset \overline{A}\cap V$.
It is sufficient to prove that $\overline{A}\cap V\subset \overline{A\cap V}$.
Let $a\in \overline{A}\cap V$.
Since $a\in \overline{A}$, there is a sequence $(a_k)_{k\in\N}\subset A$ such that $\lim\limits_{k\to\infty}a_k=a$.
Since $a\in V$, $a$ is not a root of any $p_i$.
By Lemma \ref{lem:chage.sign}, each $p_i(a)$ is positive or negative.
By the continuity of $p_i$, $p_i(a)=\lim\limits_{k\to\infty}p_i(a_k)\ge 0$ implies $a\in \{p_i>0\}$.  
Then $a$ is in $A$, and therefore, is in $A\cap V$.
Thus $a\in\overline{A\cap V}$ yields the result.
\end{proof}

\subsection{Elimination of common roots for polynomial (in)equalities}
%This section shows how to decompose a semi-algebraic set as the union of a finite number of elementary semi-algebraic sets with properties as in Lemma \ref{lem:closure}.
Let $B$ be an elementary semi-algebraic subset of $\R$ defined by
\begin{equation}\label{eq:elem.semial.set.uni}
B = \{t\in\R\,:\,p_1(t)> 0,\dots,p_m(t)> 0\,,\,q_1(t)=\dots=q_l(t)=0\}\,,
\end{equation}
where $p_i,q_j$ are polynomials in $\R[t]$.
For simplicity, we write $B$ as
\begin{equation}\label{eq:elem.semial.set.uni.simple}
B = \{p_1> 0,\dots,p_m>0\,,\,q_1=\dots=q_l=0\}\,.
\end{equation}
The set $\{q_1=\dots=q_l=0\}$ is call variety defined by polynomials $q_1,\dots,q_l$.

The following two lemmas allow us to obtain a new semi-algebraic description of an elementary semi-algebraic set $B$ such that the new inequality polynomials have no root in the variety defined by the new equality polynomials determining $B$:
\begin{lemma}\label{lem:no.root.variety}
Let $B\subset \R$ be an elementary semi-algebraic set  defined as in \eqref{eq:elem.semial.set.uni.simple} such that each $q_j$ has only simple roots.
Let $i\in\{1,\dots,m\}$ be fixed.
Set $w=\gcd(p_i,q_1,\dots,q_l)$.
For $j=1,\dots,l$, set $\check q_j=q_j/w$.
Then the following two statements hold:
\begin{enumerate}
\item $B=\{p_1>0,\dots,p_m>0\,,\,\check q_1=\dots=\check q_l=0\}$.
\item $p_i$ has no root in $\{\check q_1=\dots=\check q_l=0\}$.
\end{enumerate}
\end{lemma}
\begin{proof}
Set $A=\{p_1>0,\dots,p_m>0\,,\,\check q_1=\dots=\check q_l=0\}$.
Let $z\in B$. If $w(z)=0$, then $p_i(z)=0$ by assumption. It is impossible since $p_i(z)>0$.
Thus $w(z)\ne 0$.
Since $w(z)\check q_j(z)=q_j(z)=0$, $\check q_j(z)=0$ gives $z\in A$.
It implies $B\subset A$.
It is clear that $A\subset B$ since $\check q_j(z)=0$ implies $q_j(z)=0$.
Thus $A=B$ yields the first statement.
Assume by contradiction that $p_i$ has root $a\in \{\check q_1=\dots=\check q_l=0\}$. 
Then $q_j(a)=w(a)\check q_j(a)=0=p_i(a)$ leads to $t-a$ divides $w=\gcd(p_i,q_1,\dots,q_l)$. 
Note that $t-a$ also divides $\check q_j$.
It implies that $(t-a)^2$ divides $q_j=w\check q_j$. It is impossible since $q_j$ has only simple roots.
Hence the second statement follows.
\end{proof}
\begin{remark}\label{rem:complex.no.root.variety}
Obtaining $\check q_j$s in Lemma \ref{lem:no.root.variety} requires $(l+1)\times O(d^2)$ arithmetic operations, where $d$ is the upper bound on the degrees of $p_i$s and $q_j$s.
\end{remark}
In the following lemma, we generalize the result in Lemma \ref{lem:no.root.variety} by removing the assumption that each inequality polynomial has only simple roots:
\begin{lemma}\label{lem:no.root.variety2}
Let $B\subset \R$ be an elementary semi-algebraic set  defined as in \eqref{eq:elem.semial.set.uni.simple}.
Then there is a symbolic algorithm which produces $\check q_1,\dots,\check q_l$ such that the following two statements hold:
\begin{enumerate}
\item $B=\{p_1>0,\dots,p_m>0\,,\,\check q_1=\dots=\check q_l=0\}$.
\item For $i=1,\dots,m$, $p_i$ has no root in $\{\check q_1=\dots=\check q_l=0\}$.
\end{enumerate}
\end{lemma}
\begin{proof}
For $j=1,\dots,l$, set $\hat q_j=q_j/\gcd(q_j,q_j')$.
Lemma \ref{lem:radical} says that $\hat q_j$ has only simple roots and $\{\hat q_j=0\}=\{q_j=0\}$.
It implies that 
\begin{equation}
B=\{p_1>0,\dots,p_m>0\,,\,\hat q_1=\dots=\hat q_l=0\}\,.
\end{equation}
Applying Lemma \ref{lem:no.root.variety} several times with $i=1,\dots,m$, we obtain $\check q_1,\dots,\check q_l$ as in the conclusion.
\end{proof}
\begin{remark}\label{rem:complex.no.root.variety2}
Thanks to Remark \ref{rem:complex.no.root.variety},
obtaining $\check q_j$s in Lemma \ref{lem:no.root.variety2} requires $l+(m+1)(l+1)\times O(d^2)$ arithmetic operations, where $d$ is the upper bound on the degrees of $p_i$s and $q_j$s.
\end{remark}
The following lemma is an induction step in the proof of later Lemma \ref{lem:alg.remove.common.root}:
\begin{lemma}\label{lem:remove.common.root}
Let $B\subset \R$ be an elementary semi-algebraic set  defined as in \eqref{eq:elem.semial.set.uni.simple}.
Let $s<r<m$ and assume that the following conditions hold:
\begin{enumerate}
\item For $i=1,\dots,m$, polynomial $p_i$ has only simple roots.
\item For $i=1,\dots,m$, polynomial $p_i$ has no root in $\{q_1=\dots=q_l=0\}$.
\item For every $(i,j)\in\{1,\dots,r\}^2$ satisfying $i\ne j$, polynomials $p_i$ and $p_j$ have no common root.
\item For $i=1,\dots,s$, polynomials $p_{r+1}$ and $p_i$ have no common root.
\end{enumerate}
Set
$u_0:=\gcd(p_{s+1},p_{r+1})$.
Let $u_{s+1}=p_{s+1}/u_0$ and $u_{r+1}=p_{r+1}/u_0$.
Then $B=B_+\cup B_-$, where 
\begin{equation}
\begin{array}{rl}
B_+ = &\{u_0>0,p_1> 0,\dots,p_s>0,u_{s+1}>0,p_{s+2}>0,\dots,\\
&\ \ p_{r}>0,u_{r+1}>0,p_{r+2}>0,\dots,p_m>0\,,\,q_1=\dots=q_l=0\}\text{ and}\\
B_- = &\{-u_0>0,p_1> 0,\dots,p_s>0,-u_{s+1}>0,p_{s+2}>0,\dots,\\
&\ \ p_{r}>0,-u_{r+1}>0,p_{r+2}>0,\dots,p_m>0\,,\,q_1=\dots=q_l=0\}\\
\end{array}
\end{equation}
are elementary semi-algebraic sets such that the following statements hold:
\begin{enumerate}
\item Polynomials $u_0,u_{s+1},u_{r+1}\in\R[t]$ have only simple roots.
\item Polynomials $u_0,u_{s+1},u_{r+1}\in\R[t]$ have no root in $\{q_1=\dots=q_l=0\}$.
\item For $i\in\{1,\dots,s\}\cup\{s+2,\dots,r\}$, polynomials $u_0$ and $p_i$ have no common root.
\item For $i\in\{1,\dots,s\}\cup\{s+2,\dots,r\}$, polynomials $u_{s+1}$ and $p_i$ have no common root.
\item Polynomials $u_0$ and $u_{s+1}$ have no common root;
\item For $i=1,\dots,s$, polynomials $u_{r+1}$ and $p_i$ have no common root.
\item Polynomials $u_{r+1}$ and $u_0$ have no common root;
\item  Polynomials $u_{r+1}$ and $u_{s+1}$ have no common root.
\end{enumerate}
\end{lemma}
\begin{proof}
Since $p_{s+1}=u_{s+1}u_0$ and $p_{r+1}=u_{r+1}u_0$ have only simple roots and have no root in $\{q_1=\dots=q_l=0\}$, the first two statement hold.
In addition, we get
\begin{equation}
\begin{array}{rl}
&\{p_{s+1}>0,p_{r+1}>0\}\\
=&\{u_0>0,u_{s+1}>0,u_{r+1}>0\}\cup \{-u_0>0,-u_{s+1}>0,-u_{r+1}>0\}\,.
\end{array}
\end{equation}
It implies that $B=B_+\cup B_-$.
The next two statements follow since $p_{s+1}=u_{s+1}u_0$ and the couple $p_{s+1},p_i$ has no common root, for $i\in\{1,\dots,s\}\cup\{s+2,\dots,r\}$.
The fifth statement is due to the fact that $p_{s+1}=u_{s+1}u_0$ has only simple roots.
The sixth statement is because for every $i=1,\dots,s$, $p_{r+1}=u_{r+1}u_0$ and $p_i$ have no common root.
The seventh statement is due to the fact that $p_{r+1}=u_{r+1}u_0$ has only simple roots.
If $w=\gcd(u_{s+1},u_{r+1})$ is not constant, $wu_0$ divides both $p_{s+1}$ and $p_{r+1}$.
This is impossible since $u_0=\gcd(p_{s+1},p_{r+1})$.
Thus the eighth statement holds.
\end{proof}
\begin{remark}
We interpret the result of Lemma \ref{lem:remove.common.root} more simply as follows:
The elementary semi-algebraic sets $B_+$ and $B_-$ in Lemma \ref{lem:remove.common.root} are of the form
\begin{equation}
\{\hat p_0> 0,\dots,\hat p_m>0\,,\,q_1=\dots=q_l=0\}\,,
\end{equation}
which satisfies the following conditions:
\begin{enumerate}
\item For $i=0,1,\dots,m$, polynomial $\hat p_i$ has only simple roots.
\item For $i=0,1,\dots,m$, polynomial $\hat p_i$ has no root in $\{q_1=\dots=q_l=0\}$.
\item For every $(i,j)\in\{0,1,\dots,r\}^2$ satisfying $i\ne j$, polynomials $\hat p_i$ and $\hat p_j$ have no common root.
\item For $i=0,1,\dots,s+1$, polynomials $\hat p_{r+1}$ and $\hat p_i$ have no common root.
\end{enumerate}
\end{remark}
\begin{remark}\label{rem:complex.remove.common.root}
Obtaining $u_0,u_{s+1},u_{r+1}$ in Lemma \ref{lem:remove.common.root} requires $O(d^2)+2$ arithmetic operations, where $d$ is the upper bound on the degrees of $p_i$s and $q_j$s.
\end{remark}

\subsection{Nichtnegativstellens\"atze for univariate polynomials}
Let $\Sigma[t]$ stand for the cone of sums of squares of polynomials in $\R[t]$.
For given $g\in\R[t]$, $p=\{p_1,\dots,p_m\}\subset \R[t]$ and $q=\{q_1,\dots,q_l\}\subset \R[t]$, define:
\begin{equation}
S(p,q):=\{t\in\R\,:\,p_1(t)\ge 0,\dots,p_m(t)\ge 0\,,\,q_1(t)=\dots=q_l(t)=0\}\,,
\end{equation}
and for all $k\in\N$,
\begin{equation}\label{eq:preordering}
T_k(p,q)[t]:=\left\{\sum\limits_{\alpha\in\{0,1\}^m} \sigma_\alpha p^\alpha+\sum_{j=1}^l \psi_jq_j\left|\begin{array}{rl}
&\sigma_i\in\Sigma[t]\,,\,\psi_j\in\R[t]\,,\\
&\deg(\sigma_\alpha p^\alpha)\le 2k,\deg(\psi_jq_j)\le 2k
\end{array}\right.
\right\}\,,
\end{equation}
\begin{equation}\label{eq:SOS.stenthening}
\rho_k(g,p,q):=\sup\{\lambda\in\R\,:\,g-\lambda\in T_k(p,q)[t]\}\,,
\end{equation}
where $p^\alpha:=p_1^{\alpha_1}\dots p_m^{\alpha_m}$.

Originally developed by Lasserre in \cite{lasserre2001global}, problem \eqref{eq:SOS.stenthening} is called the sums-of-squares strengthening of order $k$ for polynomial optimization problem
\begin{equation}\label{eq:pop2}
\rho^\star(g,p,q):=\inf_{t\in S(p,q)}g(t)\,.
\end{equation}
Moreover, we can formulate problem \eqref{eq:SOS.stenthening} as a semidefinite program, the optimization of a linear function over the intersection of an affine subspace with the cone of positive semidefinite matrices.
%While Powers' book \cite[Chapter 8]{powers2021certificates} provides theoretical results related to \eqref{eq:SOS.stenthening}, the work of Magron, Safey El Din, and Schweighofer is more concerned with the computation of $\rho_k(h,p,q)$.

\begin{lemma}\label{lem:cond.cert}
Let $p=\{p_1,\dots,p_m\}$ and $q=\{q_1,\dots,q_l\}$ be subsets of $\R[t]$ such that the following conditions hold:
\begin{enumerate}
\item For $i=1,\dots,m$, polynomial $p_i$ has only simple roots.
\item For every $(i,j)\in\{1,\dots,m\}^2$ satisfying $i\ne j$, polynomials $p_i$ and $p_j$ have no common root.
\end{enumerate}
Then one of the following two cases occurs:
\begin{enumerate}
\item The set $S(p,q)$ is a finite union of isolated points.
\item The set $S(p,q)$ is a finite union of closed intervals and has no isolated point. 
In this case, for each endpoint $a$ of $S(p,q)$, there exists $i\in\{1,\dots,m\}$ such that $(t-a)$ divides $p_i$ but $(t-a)^2$ does not.
\end{enumerate}
\end{lemma}
\begin{proof}
Consider the following two cases:
\begin{itemize}
\item Case 1: $\{q_1=\dots=q_l=0\}\ne \emptyset$. Then $S(p,q)\subset \{q_1=\dots=q_l=0\}$ is a finite union of isolated points.
\item Case 2: $\{q_1=\dots=q_l=0\}= \emptyset$. Then we get $S(p,q)=S(p,\emptyset)$. 
By the third statement of Lemma \ref{lem:properties.semi-algebraic.set}, $S(p,\emptyset)$ is a finite union of closed intervals and points.
Assume by contradiction that $S(p,\emptyset)$ has an isolated point $a$.
Then $a$ is a root of $p_j$ for some $j\in\{1,\dots,m\}$. 
By assumption, $a$ is a root of $p_j$ with odd multiplicity.
Lemma \ref{lem:chage.sign} yields that $p_j$ changes sign at $a$.
Then $[a,a+\varepsilon]$ or $[a-\varepsilon,a]$ is a subset of $\{p_j\ge 0\}$ for sufficiently small $\varepsilon>0$.
By assumption, $a$ is not a root of $p_i$ for any $i\in\{1,\dots,m\}\backslash \{j\}$.
It implies that $p_i(a)>0$, for $i\in\{1,\dots,m\}\backslash \{j\}$. 
Then we take $\varepsilon$ small enough such that 
\begin{equation}
\arraycolsep=1.4pt\def\arraystretch{.5}
[a-\varepsilon,a+\varepsilon]\subset \bigcap_{\begin{array}{cc}
\scriptstyle i=1 \\
 \scriptstyle i\ne j
\end{array}}^{m} \{p_i\ge 0\}\,.
\end{equation}
Thus $[a,a+\varepsilon]$ or $[a-\varepsilon,a]$ is a subset of $S(p,\emptyset)=\bigcap_{i=1}^m \{p_i\ge 0\}$.
It is impossible since $a$ is an isolated point of $S(p,\emptyset)$.
Thus $S(p,q)$ has no isolated point.
Assume that $a$ is an endpoint of $S(p,q)$. 
Then there exists $i\in\{1,\dots,m\}$ such that $p_i(a)=0$.
It implies that $(t-a)$ divides $p_i$.
By the first condition, $(t-a)^2$ does not divide $p_i$.
\end{itemize}
Hence the result follows.
\end{proof}

The following lemma follows from Nie's Nichtnegativstellensatz in \cite[Theorem 4.1]{nie2013polynomial}:

\begin{lemma}\label{lem:cert.finite.points}
Let $h\in\R[t]$, $p=\{p_1,\dots,p_m\}\subset \R[t]$, $q=\{q_1,\dots,q_l\}\subset \R[t]$.
Assume that $S(p,q)$ is a finite union of isolated points and $h$ is nonnegative on $S(p,q)$.
Then there exists $K\in\N$ such that for all $k\ge K$, $h\in T_k(p,q)[t]$.
\end{lemma}

The following lemma restates Kuhlmann--Marshall--Schwartz's Nichtnegativstellensatz in \cite[Corollary 3.3]{kuhlmann2005positivity}:
\begin{lemma}\label{lem:certificate}
Let $p=\{p_1,\dots,p_m\}$ be a subset of $\R[t]$ and $q=\emptyset$ such that $S(p,q)=\{p_1>0,\dots,p_m>0\}$ is compact and has no isolated points. 
Then 
\begin{equation}
\{h\in\R[t]\,:\,h\ge 0\text{ on }S(p,q)\}=\bigcup_{k=1}^\infty T_k(p,q)[t]\,,
\end{equation}
if and only if, for each endpoint $a$ of $S(p,q)$, there exists
$i\in\{1,\dots,m\}$ such that $(t-a)$ divides $p_i$ but $(t-a)^2$ does not.
\end{lemma}
The following lemma is a consequence of Lemmas \ref{lem:cert.finite.points} and \ref{lem:certificate}:
\begin{lemma}\label{lem:exact}
Let $g\in\R[t]$, $p=\{p_1,\dots,p_m\}\subset \R[t]$, $q=\{q_1,\dots,q_l\}\subset \R[t]$, and $\rho^\star(g,p,q)$ be as in \eqref{eq:pop2}.
Let the assumption of Lemma \ref{lem:cond.cert} hold.
Assume that $S(p,q)$ is compact.
Then there exists $K\in\N$ such that for all $k\ge K$, $g-\rho^\star(g,p,q)\in T_k(p,q)[t]$ and therefore $\rho_k(g,p,q)=\rho^\star(g,p,q)$.
\end{lemma}
\begin{proof}
By assumption, $g-\rho^\star(g,p,q)$ is nonnegative on $S(p,q)$.
Due to Lemma \ref{lem:cond.cert}, one of the two cases mentioned in this lemma occurs.
In the first case, we apply Lemma \ref{lem:cert.finite.points} to get $K\in\N$ satisfying that for all $k\ge K$, $g-\rho^\star(g,p,q)\in T_k(p,q)[t]$.
In the second case, Lemma \ref{lem:certificate} allows us to obtain such nonnegative integer $K$.
Hence the conclusion follows.
\end{proof}

\section{Products of linear factors occurring odd times}
\label{sec:prod.linear.factors.odd}

The following algorithm allows us to obtain the product of linear factors of a polynomial  occurring odd times:
\begin{algorithm}\label{alg:prod.odd.multiplicity} 
Products of linear factors occurring odd times.
\begin{itemize}
\item Input: Polynomial $p$ in $\R[t]$.
\item Output: Rational function $u_{i^\star}$ in single-variable $t$.
\end{itemize}
\begin{enumerate}
\item Set $u_0:=1$, $h_0:=p$, and $i=0$.
\item Set $q_i:=\gcd(h_i,h_i')$ and $w_i:=\gcd(q_i,q_i')$.
\item Update $u_{i+1}:=u_i\times \frac{h_iw_i}{q_i^2}$.
\item If $w_i$ is not constant, update $h_{i+1}:=w_i$, $i:=i+1$ and run again Steps 2, 3, 4. Otherwise, set $i^\star=i+1$ and stop.
\end{enumerate}
\end{algorithm}

\begin{lemma}\label{lem:prod.odd.multiplicity}
Let $u_{i^\star}$ be the output of Algorithm \ref{alg:prod.odd.multiplicity} with input $p$ as in \eqref{eq:factor.poly}.
Then there exists a constant $C\in\R$ such that
\begin{equation}\label{eq:prod.odd.multiplicity}
\arraycolsep=1.4pt\def\arraystretch{.5}
u_{i^\star}=C\times\prod_{\begin{array}{cc}
\scriptstyle i=1 \\
 \scriptstyle  m_i\text{ is odd}
\end{array}}^{s} (t-a_i)
\end{equation}
is a polynomial in $\R[t]$.
\end{lemma}
\begin{proof}
Let us prove that Algorithm \ref{alg:prod.odd.multiplicity} terminates after a finite number of steps.
By Lemma \ref{lem:mul.one}, the first two steps produce $\frac{h_0w_0}{q_0^2}$ as the right-hand side of \eqref{eq:prod.mul.one}.
It is the product of linear factors of $p$ occurring once and thus is a polynomial.
In addition, Lemma \ref{lem:radical} yields
\begin{equation}
\arraycolsep=1.4pt\def\arraystretch{.5}
w_0=\prod_{\begin{array}{cc}
\scriptstyle i=1 \\
 \scriptstyle  m_i>1
\end{array}}^{s} (t-a_i)^{m_i-2}
\end{equation}
It implies $\deg(w_0)\le \deg(h_0)-2s$.
The third step gives $u_1$ as the right-hand side of \eqref{eq:prod.mul.one}.
In the fourth step, we obtain $h_1=w0$, which gives $\deg(h_1)=\deg(w_0)\le \deg(h_0)-2s$.
Repeating this process, we get $\frac{h_iw_i}{q_i^2}$ as the product of linear factors of $p$ with multiplicity $2i+1$ and $\deg(h_{i+1})< \deg(h_i)<\deg(p)$.
Then $u_{i+1}$ is the product of linear factors of p occurring $1,3,\dots,2i+1$ times.
Moreover, Algorithm \ref{alg:prod.odd.multiplicity} cannot run forever since $\deg(h_{i+1})< \deg(h_i)<\deg(p)$.
Hence the result follows.
\end{proof}
\begin{remark}\label{rem:complex.prod.odd.multiplicity}
Thanks to the complexity given in Remark \ref{rem:complex.gcd}, Algorithm \ref{alg:prod.odd.multiplicity} with input $p$ of degree $d$ has complexity $O(d^3)$.
\end{remark}
Given a polynomial $p\in\R[t]$, denote by $\lc(p)$ the leading coefficient of $p$.
\begin{definition}\label{def:p.tilde}
For a given polynomial $p$ in $\R[t]$ with decomposition \eqref{eq:factor.poly},
we define polynomial $\widetilde p\in\R[t]$ by
\begin{equation}
\widetilde p:=\begin{cases}
\frac{\lc(p)}{\lc(u_{i^\star})}\times u_{i^\star}&\text{if }u_{i^\star}\ne 0\,,\\
0&\text{otherwise}\,.
\end{cases}
\end{equation}
where $u_{i^\star}$ is the output of Algorithm \ref{alg:prod.odd.multiplicity} with input $p$.
\end{definition}
By definition, $\widetilde p$ has the same leading coefficient as $p$.
Roughly speaking, polynomial $\widetilde p$ enables us to remove all linear factors of a polynomial $p$ occurring even times.

The following lemma states some properties of polynomial $\widetilde p$ needed for the later results:
\begin{lemma}\label{lem:properties}
Let $p$ be polynomial in $\R[t]$ with decomposition \eqref{eq:factor.poly}.
The following statements hold:
\begin{enumerate}
\item The polynomial $\widetilde p$ is the right-hand side of \eqref{eq:prod.odd.multiplicity}. Consequently, $\widetilde p$ has only simple roots.
\item The polynomial $p$ can be decomposed as $p=\widetilde p\times w$, where
\begin{equation}
\arraycolsep=1.4pt\def\arraystretch{.5}
w=\prod_{\begin{array}{cc}
\scriptstyle j=1 \\
 \scriptstyle  m_j\text{ is even}
\end{array}}^{s} (t-a_j)^{m_j} \times \prod_{\begin{array}{cc}
\scriptstyle i=1 \\
 \scriptstyle  m_i\text{ is odd}
\end{array}}^{s} (t-a_i)^{m_i-1}
\end{equation}
is a nonnegative polynomial in $\R[t]$.
Consequently, $p(t)>0$ if and only if $\widetilde p(t)>0$.
\end{enumerate}
\end{lemma}
\begin{proof}
The first statement is due to Lemma \ref{lem:prod.odd.multiplicity}.
The second statement relies on the first statement and the decomposition \eqref{eq:factor.poly}. 
It is clear that $w\in\R[t]$ thanks to the second statement of Lemma \ref{lem:factor.poly}.
The non-negativity of $w$ is because $w$ is the product of the squares of linear factors of $p$.
\end{proof}
\section{Replacement of strict inequalities by non-strict ones}
\label{sec:replace.strict.ineq}
In this section, we demonstrate that replacing strict inequalities defining a ``nice'' elementary semi-algebraic set with non-strict ones allows us to obtain the closure of this set.

\begin{definition}\label{def:closure}
For a given elementary semi-algebraic set $B\subset\R$ defined as in \eqref{eq:elem.semial.set.uni}, we define elementary semi-algebraic set $\widetilde B\subset\R$ by replacing strict inequalities $p_i(t)>0$ determining $B$ with non-strict inequalities $\widetilde p_i(t)\ge 0$, i.e.,
\begin{equation}\label{eq:replace.ineuality}
\widetilde B = \{t\in\R\,:\,\widetilde p_1(t)\ge 0,\dots,\widetilde p_m(t)\ge 0\,,\,q_1(t)=\dots=q_l(t)=0\}\,.
\end{equation}
\end{definition}
\begin{remark}
The set $\widetilde B$ in \eqref{eq:replace.ineuality} is  semi-algebraic since it can be written as
\begin{equation}
\widetilde B=\bigcup_{*\in\{>,=\}^m}\{t\in\R\,:\,\widetilde p_1(t)*_1 0,\dots,\widetilde p_m(t) *_m 0\,,\,q_1(t)=\dots=q_l(t)=0\}\,.
\end{equation}
\end{remark}
\begin{lemma}\label{lem:alternative.rep}
Let $B\subset \R$ be an elementary semi-algebraic set  defined as in \eqref{eq:elem.semial.set.uni.simple}.
Then 
\begin{equation}\label{eq:alternative.rep}
B = \{\widetilde p_1> 0,\dots,\widetilde p_m>0\,,\,q_1=\dots=q_l=0\}\,.
\end{equation}
\end{lemma}
\begin{proof}
By using the second statement of Lemma \ref{lem:properties}, we get the equivalence: $p_i(t)>0$ if and only if $\widetilde p_i(t)>0$.
Hence the result follows. 
\end{proof}
The following example given in \cite[Section 2.1]{basu1999new} shows that replacing strict inequalities in elementary semi-algebraic set $B$ without removing the factors of the inequality polynomials occurring even times might not allow for the closure of $B$ to be obtained:
\begin{example}\label{exam:even.mul}
Let $B=\{t\in\R\,:\,t^2(t-1)> 0\}$.
Then $B=(1,\infty)$.
Replacing the strict inequality in $B$ by a non-strict inequality, we obtain $\hat B=\{t\in\R\,:\,t^2(t-1)\ge 0\}=\{0\}\cup [1,\infty)$, which is not the closure of $B$.
However, by definition, $\widetilde B=\{t\in\R\,:\,t-1\ge 0\}=[1,\infty)$ is exactly the closure of $B$.
\end{example}

The following example indicates the case where $\widetilde B$ is not the closure of an elementary semi-algebraic set $B$:
\begin{example}\label{exam:cl.intersection}
Let $B=\{p_1>0,p_2>0\}$ with $p_1=-t(t-1)$ and $p_2=-t(t+1)$. 
Then $\{p_1>0\}=(0,1)$ and $\{p_2\ge 0\}=(-1,0)$ yields $B=\{p_1>0\}\cap \{p_2\ge 0\}=\emptyset$.
However, since $\widetilde p_j=p_j$,  
$\{\widetilde p_1>0\}=[0,1]$ and $\{\widetilde p_2\ge 0\}=[-1,0]$ yields $\widetilde B=\{\widetilde p_1>0\}\cap \{\widetilde p_2\ge 0\}=\{0\}$. Thus $\overline{B}=\emptyset\ne \{0\}=\widetilde B$.
\end{example}

To address the issues in Examples \ref{exam:even.mul} and \ref{exam:cl.intersection}, we need the following lemma:
\begin{lemma}\label{lem:closure}
Let $B\subset\R$ be an elementary semi-algebraic set of the form \eqref{eq:elem.semial.set.uni.simple}.
Suppose that the following conditions hold:
\begin{enumerate}
\item Any couple of inequality polynomials defining $B$ have no common real root with odd multiplicity.
\item Any inequality polynomial defining $B$ has no  root in $\{q_1=\dots=q_l=0\}$ with odd multiplicity.
\end{enumerate} 
Then $\widetilde B$ is the closure of $B$, i.e., $\widetilde  B=\overline{B}$.
\end{lemma}
\begin{proof}
By Lemma \ref{lem:alternative.rep}, equality \eqref{eq:alternative.rep} holds.
For $i=1,\dots,m$, let $a_{i,1},\dots,a_{i,r_i}$ be the  real roots of $p_i$ with odd multiplicities.
Lemma \ref{lem:properties} yields that $a_{i,1},\dots,a_{i,r_i}$ are all simple real roots of $\widetilde p_i$.
By assumption, for every $(i,r)\in\{1,\dots,m\}^2$ satisfying $i\ne j$, $\widetilde p_i$ and $\widetilde p_r$ have no common root. 
From this, Lemma \ref{lem:closure.strict} yields 
\begin{equation}\label{eq:cl.union}
\overline{\bigcap_{i=1}^m \{\widetilde p_i>0\}}=\bigcap_{i=1}^m \{\widetilde p_i\ge 0\}\,.
\end{equation}
Set $V=\{q_1=\dots=q_l=0\}$. 
By assumption, each $\widetilde p_i$ has no root in $V$.
From this, Lemma \ref{lem:closure.strict.intersec} implies 
\begin{equation}
\overline{\bigcap_{i=1}^m \{\widetilde p_i>0\}}\cap V=\overline{\bigcap_{i=1}^m \{\widetilde p_i>0\}\cap V}\,.
\end{equation}
Combining this with \eqref{eq:cl.union}, \eqref{eq:alternative.rep}, and \eqref{eq:replace.ineuality}, we obtain the result.
\end{proof}

Although not all elementary semi-algebraic sets have properties as in Lemma \ref{lem:closure}, we can always decompose each elementary semi-algebraic set as the union of a finite number of elementary semi-algebraic sets with these properties:
\begin{lemma}\label{lem:alg.remove.common.root}
Let $B\subset\R$ be an elementary semi-algebraic set.
Then the is a symbolic algorithm which produces elementary semi-algebraic subsets $B_1,\dots,B_s$ of $\R$ 
 such that $B=\bigcup_{j=1}^sB_j$ and 
the following two statements hold:
\begin{enumerate}
\item For $j=1,\dots,s$, each inequality polynomial defining $B_j$ has only simple roots.
\item For $j=1,\dots,s$, any inequality polynomial defining $B_j$ has no root in the variety determined by the equality polynomials defining $B_j$.
\item For $j=1,\dots,s$, any couple of inequality polynomials defining $B_j$ has no common root.  
\end{enumerate} 
\end{lemma}
\begin{proof}
Suppose that $B$ is of the form \eqref{eq:elem.semial.set.uni.simple}.
By Lemma \ref{lem:alternative.rep}, equality \eqref{eq:alternative.rep} holds.
Using Lemma \ref{lem:no.root.variety2}, we get $\check q_j\in\R[t]$ such that
\begin{equation}
\widetilde B = \{t\in\R\,:\,\widetilde p_1(t)\ge 0,\dots,\widetilde p_m(t)\ge 0\,,\,\check q_1(t)=\dots=\check q_l(t)=0\}\,.
\end{equation}
and each $\widetilde p_i$ has no root in $\{\check q_1(t)=\dots=\check q_l(t)=0\}$.
Note that the first statement of Lemma \ref{lem:properties} implies that each $\widetilde p_i$ has only simple roots.
Applying Lemma \ref{lem:remove.common.root} several times with $s=1,\dots,r-1$ and $r=1,\dots,m-1$,
we obtain $B_1,\dots,B_s$ as in the conclusion.  
\end{proof}
\begin{remark} \label{rem:complex.alg.remove.common.root}
Thanks to Remarks \ref{rem:complex.prod.odd.multiplicity}, \ref{rem:complex.no.root.variety2}, and \ref{rem:complex.remove.common.root},
the algorithm mentioned in Lemma \ref{lem:alg.remove.common.root} has complexity 
\begin{equation}
l+(m+1)(l+1)\times O(d^2)+m\times O(d^3)+m!\times(O(d^2)+2)\,,
\end{equation}
where $B$ is of the form \eqref{eq:elem.semial.set.uni.simple} and $d$ is the upper bound on the degrees of $p_i$s and $q_j$s.
\end{remark}

\section{Main algorithm and application to polynomial optimization}
\label{sec:main.alg.app.poly.opt}
The following algorithm enables us to obtain semidefinite descriptions of the closures of the images of semi-algebraic sets under polynomials:
\begin{algorithm}\label{alg:closure.image.semi} 
Closures of the images of semi-algebraic sets under polynomials.
\begin{itemize}
\item Input: Semi-algebraic set $S\subset\R^n$ and polynomial $f\in\R[x]$.
\item Output: Semi-algebraic description of $Q\subset\R$.
\end{itemize}
\begin{enumerate}
\item Compute $p_{i,j},q_{k,j}\in\R[t]$ such that $f(S)=\bigcup_{j=1}^s B_j$ with 
\begin{equation}
B_j=\{p_{1,j}>0,\dots,p_{m_j,j}>0\,,\,q_{1,j}=\dots=q_{l_j,j}=0\}\,.
\end{equation}
\item For $j=1,\dots,s$, do:
\begin{enumerate}
\item Compute $p_{i,r,j},q_{k,r,j}\in\R[t]$ such that $B_j=\bigcup_{r=1}^{s_j} B_{r,j}$ with 
\begin{equation}
B_{r,j}=\{p_{1,r,j}>0,\dots,p_{m_{r,j},r,j}>0\,,\,q_{1,r,j}=\dots=q_{l_{r,j},r,j}=0\}
\end{equation}
and the following three conditions hold:
\begin{enumerate}
\item Each of $p_{1,r,j},\dots,p_{m_{r,j},r,j}$ has only simple root.
\item Any of $p_{1,r,j},\dots,p_{m_{r,j},r,j}$ has no root in $\{q_{1,r,j}=\dots=q_{l_{r,j},r,j}=0\}$.
\item Any couple of $p_{1,r,j},\dots,p_{m_{r,j},r,j}$ has no common root.
\end{enumerate}
\item Compute $\widetilde p_{i,r,j}$ as in Definition \ref{def:p.tilde} by using Algorithm \ref{alg:prod.odd.multiplicity}.
\end{enumerate}
\item Set $Q:=\bigcup_{j=1}^s\bigcup_{r=1}^{s_j} \widetilde B_{r,j}$ with $\widetilde B_{r,j}$ as in Definition \ref{def:closure}, i.e.,
\begin{equation}
\widetilde B_{r,j}=\{\widetilde p_{1,r,j}\ge 0,\dots,\widetilde p_{m_{r,j},r,j}\ge 0\,,\,q_{1,r,j}=\dots=q_{l_{r,j},r,j}=0\}\,.
\end{equation}
\end{enumerate}
\end{algorithm}
The first step of Algorithm \ref{alg:closure.image.semi} relies on Basu's method in \cite[Section 3]{basu1999new}.
We use the algorithm mentioned in Lemma \ref{lem:alg.remove.common.root} to do Step 2 (a) of Algorithm \ref{alg:closure.image.semi}.

We guarantee that the output $Q$ of Algorithm \ref{alg:closure.image.semi} is identical to the closure $\overline{f(S)}$ in the following theorem:

\begin{theorem}\label{theo:closure}
Let $f$ be a polynomial and $S$ be a semi-algebraic set.
Then Algorithm \ref{alg:closure.image.semi} with input $f,S$ produces polynomials defining $A_1,\dots,A_w$ of the form \eqref{eq:cl.semialg.set} such that $Q=\bigcup_{j=1}^w A_j$ is exactly the closure of $f(S)$ and the following conditions hold:
\begin{enumerate}
\item Each inequality polynomial defining $A_j$ has only simple roots.
%\item Any inequality polynomial defining $A_j$ has no root in the variety determined by the equality polynomials defining $A_j$.
\item Any couple of inequality polynomials determining each $A_j$ has no common root. 
\end{enumerate}
\end{theorem}
\begin{proof}
Lemma \ref{lem:closure} implies that $\widetilde B_{r,j}=\overline{B_{r,j}}$.
From this and the second statement of Lemma \ref{lem:intersec.closure}, we get
\begin{equation}
Q=\bigcup_{j=1}^s\bigcup_{r=1}^{s_j} \overline{B_{r,j}}=\overline{\bigcup_{j=1}^s\bigcup_{r=1}^{s_j} B_{r,j}}=\overline{\bigcup_{j=1}^s B_j}=\overline{f(S)}\,.
\end{equation}
Setting $A_1,\dots,A_w$ as $\widetilde B_{r,j}$s, we obtain the result.
\end{proof}
\begin{remark} 
Assume that $S\subset \R^n$ is determined by $s$ polynomials of degree at most $d$ and $f$ has degree at most $d$.
By Remark \ref{rem:complex.polyimage.semi}, the first step of Algorithm \ref{alg:closure.image.semi} requires $s^{n+1}d^{O(n)}$ arithmetic operations and the degrees of $p_{i,j},q_{k,j}$ are upper bounded by $d^{O(n)}$.
Remark \ref{rem:complex.alg.remove.common.root} yields that Step 2 (a) of Algorithm \ref{alg:closure.image.semi} has a complexity of
$l_j+(m_j+1)(l_j+1)\times O(d^{2\times O(n)})+m_j\times O(d^{3\times O(n)})+m_j!\times(O(d^{2\times O(n)})+2)$.
Thanks to Remark \ref{rem:complex.prod.odd.multiplicity}, the complexity of Step 2 (b) of Algorithm \ref{alg:closure.image.semi} is of $O(d^{3\times O(n)})\times \sum_{r=1}^{s_j} m_{r,j}$.
Note that the degrees of the polynomials do not increase in Steps 2 and 3 of Algorithm \ref{alg:closure.image.semi}.
Thus Algorithm \ref{alg:closure.image.semi} has the number of arithmetic operations as $O(d^{O(n)})$ to produce polynomials of degrees at most $d^{O(n)}$ that define the closure $\overline{f(S)}$.
\end{remark}

The following corollary makes use of Algorithm \ref{alg:closure.image.semi} in polynomial optimization, especially in the case where the objective polynomials do not attain minimum value on the semi-algebraic sets:
\begin{corollary}\label{coro:application}
Let $f$ be a polynomial and $S$ be a semi-algebraic set.
Consider polynomial optimization problem:
\begin{equation}
f^\star=\inf_{x\in S}f(x)\,.
\end{equation}
Assume that the image $f(S)$ is bounded.
Let $Q$ be as the output of Algorithm \ref{alg:closure.image.semi} with input $f,S$.
Then there exist $p^{(j)}\subset \R[t]$ and $p^{(j)}\subset\R[t]$ such that
 \begin{equation}
f^\star=\min_{t\in Q}g(t)\,,
\end{equation}
where $g(t)=t$ and $Q=\bigcup_{j=1}^w S(p^{(j)},q^{(j)})$ satisfying the following conditions:
\begin{enumerate}
\item Each polynomial in $p^{(j)}$ has only simple roots.
\item Any couple of polynomials in $p^{(j)}$ has no common root.
\end{enumerate} 
Moreover, there exists $K\in\N$ such that for all $k\ge K$,
 \begin{equation}
f^\star=\min\limits_{j=1,\dots,w}\rho_k(g,p^{(j)},q^{(j)})\,.
\end{equation}
\end{corollary}
\begin{proof}
Theorem \ref{theo:closure} yields $Q=\overline{f(S)}$, which implies the first statement.
Let us prove the second one.
Since $f(S)$ is bounded, so is $S(p^{(j)},p^{(j)})$.
From $Q=\bigcup_{j=1}^w S(p^{(j)},p^{(j)})$, we get
\begin{equation}\label{eq:min.rep}
f^\star=\min\limits_{j=1,\dots,w}\rho^\star(g,p^{(j)},q^{(j)})\,.
\end{equation}
Lemma \ref{lem:exact} says that there exists $K_j\in\N$ such that for all $k\ge K_j$, $\rho_k(g,p^{(j)},q^{(j)})=\rho^\star(g,p^{(j)},q^{(j)})$.
Set $K=\max\limits_{j=1,\dots,w}K_j$. 
Hence the result follows thanks to \eqref{eq:min.rep}.
\end{proof}
\if{
\section{Nichtnegativstellens\"atze for multivariate polynomials}
\label{sec:representation}
Let $\Sigma[x]$ stand for the cone of sums of squares of polynomials in $\R[x]$.
For given $f\in\R[x]$, $p=\{p_1,\dots,p_m\}\subset \R[x]$ and $q=\{q_1,\dots,q_l\}\subset \R[x]$, for all $k\in\N$, define
\begin{equation}\label{eq:preordering.multi}
T_k(p,q)[x]:=\left\{\sum\limits_{\alpha\in\{0,1\}^m} \sigma_\alpha p^\alpha+\sum_{j=1}^l \psi_jq_j \left|\begin{array}{rl}
&\sigma_i\in\Sigma[x]\,,\,\psi_j\in\R[x]\,,\\
&\deg(\sigma_\alpha p^\alpha)\le 2k,\deg(\psi_jq_j)\le 2k
\end{array}\right.
\right\}\,,
\end{equation}
where $p^\alpha:=p_1^{\alpha_1}\dots p_m^{\alpha_m}$.
Given $f\in\R[x]$ and $u=\{u_1,\dots,u_r\}\subset \R[t]$, we write $u\circ f=\{u_1\circ f,\dots,u_r\circ f\}$.
In this case, $u\circ f\subset\R[x]$.
Note that if $f\in\R[x]$ and $\sigma\in \Sigma[t]$, then $\sigma\circ f\in\Sigma[x]$.

The following Nichtnegativstellens\"atze follow from Lemma \ref{lem:exact}:
\begin{corollary}
Let the assumption and notation of Corollary \ref{coro:application} hold.
Then there exists $K\in\N$ such that for all $k\ge K$, $f-f^\star\in T_k(p^{(j)}\circ f,q^{(j)}\circ f)[x]$, $j=1,\dots,w$.
\end{corollary}
\begin{proof}
By Lemma \ref{lem:exact}, there exists $K\in\N$ such that for all $k\ge K$, $t-\rho^\star(g,p^{(j)},p^{(j)})\in T_k(p^{(j)},p^{(j)})[t]$.
From this and \eqref{eq:min.rep}, we get
\begin{equation}
t-f^\star=(t-\rho^\star(g,p^{(j)},p^{(j)})+(\rho^\star(g,p^{(j)},p^{(j)}-f^\star)\in T_k(p^{(j)},p^{(j)})[t]\,.
\end{equation}
Setting $t=f(x)$ allows us to obtain the result.
\end{proof}
}\fi
\paragraph{Acknowledgements.} 
The author would like to thank Pham Tien Son and Dinh Si Tiep for their valuable discussions about this paper.
%We are grateful to Pham Tien Son for relevant remarks and suggestions.

%\bibliography{references} 

\begin{thebibliography}{10}

\bibitem{basu1999new}
S.~Basu.
\newblock New results on quantifier elimination over real closed fields and
  applications to constraint databases.
\newblock {\em Journal of the ACM (JACM)}, 46(4):537--555, 1999.

\bibitem{basu2003algorithms}
S.~Basu, R.~Pollack, and M.~Roy.
\newblock {\em Algorithms in Real Algebraic Geometry}.
\newblock Algorithms and computation in mathematics. Springer, 2003.

\bibitem{belhaj2013complexity}
S.~Belhaj and H.~B. Kahla.
\newblock {On the complexity of computing the GCD of two polynomials via Hankel
  matrices}.
\newblock {\em ACM Communications in Computer Algebra}, 46(3/4):74--75, 2013.

\bibitem{coste2000introduction}
M.~Coste.
\newblock An introduction to semialgebraic geometry, 2000.

\bibitem{demmel2007representations}
J.~Demmel, J.~Nie, and V.~Powers.
\newblock {Representations of positive polynomials on noncompact semialgebraic
  sets via KKT ideals}.
\newblock {\em Journal of pure and applied algebra}, 209(1):189--200, 2007.

\bibitem{dinh2014frank}
S.~T. Dinh, H.~V. Ha, and T.~S. Pham.
\newblock {A Frank--Wolfe type theorem for nondegenerate polynomial programs}.
\newblock {\em Mathematical Programming}, 147(1):519--538, 2014.

\bibitem{ha2009solving}
H.~V. H{\`a} and T.~S. Pham.
\newblock Solving polynomial optimization problems via the truncated tangency
  variety and sums of squares.
\newblock {\em Journal of Pure and Applied Algebra}, 213(11):2167--2176, 2009.

\bibitem{heintz1990complexite}
J.~Heintz, M.-F. Roy, and P.~Solern{\'o}.
\newblock {Sur la complexit{\'e} du principe de Tarski-Seidenberg}.
\newblock {\em Bulletin de la Soci{\'e}t{\'e} math{\'e}matique de France},
  118(1):101--126, 1990.

\bibitem{korkina1985another}
E.~Korkina and A.~G. Kushnirenko.
\newblock {Another proof of the Rarski-Seidenberg theorem}.
\newblock {\em Siberian Mathematical Journal}, 26(5):703--707, 1985.

\bibitem{kuhlmann2005positivity}
S.~Kuhlmann, M.~Marshall, and N.~Schwartz.
\newblock {Positivity, sums of squares and the multi-dimensional moment problem
  II}.
\newblock 2005.

\bibitem{lasserre2001global}
J.~B. Lasserre.
\newblock Global optimization with polynomials and the problem of moments.
\newblock {\em SIAM Journal on optimization}, 11(3):796--817, 2001.

\bibitem{magron2015semidefinite}
V.~Magron, D.~Henrion, and J.-B. Lasserre.
\newblock Semidefinite approximations of projections and polynomial images of
  semialgebraic sets.
\newblock {\em SIAM Journal on Optimization}, 25(4):2143--2164, 2015.

\bibitem{mai2022symbolic}
N.~H.~A. Mai.
\newblock {A symbolic algorithm for exact polynomial optimization strengthened
  with Fritz John conditions}.
\newblock {\em arXiv preprint arXiv:2206.02643}, 2022.

\bibitem{mai2022exact}
N.~H.~A. Mai.
\newblock {Exact polynomial optimization strengthened with Fritz John
  conditions}.
\newblock {\em arXiv preprint arXiv:2205.04254}, 2022.

\bibitem{mai2021positivity}
N.~H.~A. Mai, J.-B. Lasserre, and V.~Magron.
\newblock Positivity certificates and polynomial optimization on non-compact
  semialgebraic sets.
\newblock {\em Mathematical Programming}, pages 1--43, 2021.

\bibitem{nie2013polynomial}
J.~Nie.
\newblock Polynomial optimization with real varieties.
\newblock {\em SIAM Journal On Optimization}, 23(3):1634--1646, 2013.

\bibitem{nie2006minimizing}
J.~Nie, J.~Demmel, and B.~Sturmfels.
\newblock Minimizing polynomials via sum of squares over the gradient ideal.
\newblock {\em Mathematical programming}, 106(3):587--606, 2006.

\bibitem{pham2019optimality}
T.-S. Pham.
\newblock Optimality conditions for minimizers at infinity in polynomial
  programming.
\newblock {\em Mathematics of Operations Research}, 44(4):1381--1395, 2019.

\bibitem{pham2016genericity}
T.~S. Pham and H.~H. Vui.
\newblock {\em Genericity in polynomial optimization}, volume~3.
\newblock World Scientific, 2016.

\bibitem{schweighofer2006global}
M.~Schweighofer.
\newblock Global optimization of polynomials using gradient tentacles and sums
  of squares.
\newblock {\em SIAM Journal on Optimization}, 17(3):920--942, 2006.

\bibitem{seidenberg1954new}
A.~Seidenberg.
\newblock A new decision method for elementary algebra.
\newblock {\em Annals of Mathematics}, pages 365--374, 1954.

\bibitem{tarski1951decision}
A.~Tarski.
\newblock A decision method for elementary algebra and geometry, revised.
\newblock {\em Berkeley and Los Angeles}, 1951.

\end{thebibliography}
\bibliographystyle{abbrv}

\end{document}